\documentclass[11pt]{article}
\usepackage[T1]{fontenc}
\usepackage{amsmath,amsxtra,amssymb,latexsym,amscd,amsthm,pb-diagram,fancyhdr,euscript}
\usepackage{amsthm}
\usepackage{indentfirst}
\usepackage{epsfig}
\usepackage{mathrsfs}
\usepackage{slashed}
\usepackage{hyperref}
\usepackage{float}
\usepackage{cases}
\hypersetup{
   colorlinks,
    citecolor=blue,
    filecolor=black,
    linkcolor=blue,
    urlcolor=magenta
}
\hypersetup{linktocpage}

\renewcommand{\d}{\mathrm{d}}

\newcommand{\scri}{{\mathscr I}}

\newcommand{\hook}{{\setlength{\unitlength}{11pt}   
                   \begin{picture}(.833,.8)
                   \put(.15,.08){\line(1,0){.35}}
                   \put(.5,.08){\line(0,1){.5}}
                   \end{picture}}}
\newtheorem{definition}{Definition}
\newtheorem{theorem}{Theorem}

\newtheorem{corollary}{Corollary}
\newtheorem{lemma}{Lemma}
\newtheorem{remark}{Remark}

\begin{document}
\mbox{} \thispagestyle{empty}

\begin{center}
\bf{\Huge Geometric scattering for nonlinear wave equations on the Schwarzschild metric} \\

\vspace{0.1in}

{Pham Truong Xuan}\footnote{Thang Long Institute of Mathematics and Applied Sciences (TIMAS), Thang Long University, Nghiem Xuan Yem, Hoang Mai, Hanoi, Vietnam and Faculty of Mathematics and Informatics, Hanoi University of Science and Technology, 1 Dai Co Viet, Hanoi, Vietnam. Email: xuanpt@thanglong.edu.vn or phamtruongxuan.k5@gmail.com}
\end{center}

{\bf Abstract.} 
In this paper, we establish a conformal scattering theory for defocusing semilinear wave equations on Schwarzschild spacetime. We combine the energy and pointwise decay results for solutions obtained in \cite{Yang} with a Sobolev embedding on spacelike hypersurfaces to derive two-sided energy estimates between the energy flux of solutions through the Cauchy initial hypersurface $\Sigma_0 = \{ t = 0 \}$
and that through the null conformal boundaries $\mathfrak{H}^+ \cup \scri^+$ (respectively, $\mathfrak{H}^- \cup \scri^-$). By combining these estimates with the well-posedness of the Cauchy and Goursat problems for nonlinear wave equations, we construct a bounded linear and locally Lipschitz scattering operator that maps past scattering data to future scattering data.



\tableofcontents

\section{Introduction}
In the present paper, we consider the following nonlinear (defocusing semilinear) wave equation on the exterior domain $(\mathcal{B}_I, g)$ of the Schwarzschild black hole {(see Subsection~\ref{S23} for further details)}:
\begin{equation}\label{0nonlinearequation}
\Box_g\psi + |\psi|^2\psi =0,
\end{equation}
where $g$ is the Schwarzschild metric, $\Box_g$ is the scalar Laplace operator associating with the metric $g$.

We first briefly recall some important results on the scattering theory for nonlinear scalar wave equations. The first work providing an analytic scattering theory for equation~\eqref{0nonlinearequation} on Minkowski spacetime was due to Baez and Zhou in~\cite{Ba1989b}. Subsequently, Hidano~\cite{Hi1,Hi2,Hi3} established analytic scattering results for nonlinear wave equations on Euclidean space $\mathbb{R}^n$ (with $n=3,4$) for nonlinearities of the form $f(u) = |u|^{p-1}u$ with $p>1$.
On the other hand, geometric scattering (also called conformal scattering) for equation~\eqref{0nonlinearequation} on Minkowski spacetime was initiated by Baez \emph{et al.} in~\cite{Ba1989a,BaSeZho1990}. Later, Joudioux studied conformal scattering for equation~\eqref{0nonlinearequation} on {asymptotically} simple spacetimes in~\cite{Jo2012,Jo2019}. To construct the scattering operator, the authors in~\cite{BaSeZho1990,Jo2012,Jo2019} established the well-posedness of the Goursat problem for equation~\eqref{0nonlinearequation} with initial data prescribed on the conformal null boundaries of spacetime. We note that the Goursat problem was first introduced by H\"ormander for linear scalar wave equations on spatially compact spacetimes.
We now emphasize that there are two fundamental geometric ingredients underlying the {formulation} of the Goursat problem for wave equations. The first is the conformal compactification of spacetime, an idea introduced by Penrose~\cite{Pe1964}, who provided a conformal embedding of Minkowski spacetime into the Einstein {cylinder}. The second is the notion of the radiation field introduced by Friedlander~\cite{Fri1962,Fri1964,Fri1967}, which defines scattering data on the conformal null boundary. In particular, using the radiation field, Friedlander established scattering theory for linear wave equations on asymptotically Euclidean manifolds in~\cite{Fri2001}.

Conformal scattering theory on curved spacetimes was developed as a systematic research program by Mason and Nicolas, who constructed scattering theories for the linear wave, Dirac, and Maxwell equations on asymptotically simple spacetimes in~\cite{MaNi2004}. Working also on asymptotically simple spacetimes, Joudioux established conformal scattering results for the nonlinear wave equation~\eqref{0nonlinearequation} in~\cite{Jo2012,Jo2019}. Subsequently, Nicolas constructed a conformal scattering theory for linear wave equations on Schwarzschild spacetime, which is static and spherically symmetric, in~\cite{Ni2016}, while Mokdad developed the corresponding theory for the Maxwell equations on Reissner--Nordstr\"om--de~Sitter spacetime in~\cite{Mo2019}. Extending the works of Nicolas and Mokdad, Pham constructed conformal scattering theories for scalar wave equations with potentials and for tensorial wave equations on Schwarzschild spacetime in~\cite{Pha2019,Pha2023}, respectively. Furthermore, Pham established a conformal scattering theory for the Dirac equation on Kerr spacetime in~\cite{Pha2022}. In addition, Taujanskas constructed a conformal scattering theory for the Maxwell--scalar field system on de~Sitter spacetime in~\cite{Ta2019}. To the best of our knowledge, there is no existing work that addresses the scattering theory (either analytic or conformal) for equation~\eqref{0nonlinearequation} on Schwarzschild spacetime.

In the present paper, we construct a conformal scattering theory for equation~\eqref{0nonlinearequation} on Schwarzschild spacetime. We describe our strategy as follows. By using a conformal transformation, we obtain the conformally rescaled wave equation on the Penrose compactification $(\bar{\mathcal{B}}_I, \hat{g})$ of the exterior region of the Schwarzschild black hole. We consider a foliation $\{ \mathcal{S}_\tau \}_{\tau}$ of the future domain $\mathcal{I}^+(\Sigma_0)$, where each hypersurface $\mathcal{S}_\tau$ consists of two parts: a spacelike part for $2M \leq r \leq r_{FH}$ and a null part for $r \geq r_{FH}$. Using the energy and pointwise decay estimates for solutions to nonlinear wave equations recently obtained by Yang \emph{et al.} in~\cite{Yang}, we prove that the energy flux of solutions through $\mathcal{S}_T$ tends to zero as $T \to +\infty$. This yields an energy equality between the energy on $\Sigma_0$ and that on $\mathfrak{H}^+ \cup \mathscr{I}^+$ (see Theorem~\ref{equality1}). By employing Sobolev embeddings on spacelike hypersurfaces, we further establish two-sided energy estimates (without higher-order terms) between the energy flux on $\Sigma_0$ and that on $\mathfrak{H}^+ \cup \mathscr{I}^+$, as well as on $\Sigma_t$ (see Corollary~\ref{1} and Theorem~\ref{equality2}). These estimates allow us to prove the well-posedness of the Cauchy and Goursat problems in Section~\ref{S4}. In particular, using the two-sided energy estimates between $\Sigma_0$ and $\Sigma_t$ together with the method in~\cite{CaCho}, we obtain the well-posedness of the Cauchy problem for nonlinear wave equations on $\bar{\mathcal{B}}_I$ (see Theorem~\ref{CauchyProblem}). As a consequence, we define the trace operator $\mathbb{T}^+$, which maps initial data with smooth compact support on $\Sigma_0$ to scattering data with smooth compact support on $\mathfrak{H}^+ \cup \mathscr{I}^+$, namely
$\mathbb{T}^+ : \mathcal{C}_0^\infty(\Sigma_0) \times \mathcal{C}_0^\infty(\Sigma_0)
\longrightarrow
\mathcal{C}_0^\infty(\mathfrak{H}^+) \times \mathcal{C}_0^\infty(\mathscr{I}^+)$.
This operator extends continuously to the energy space as
$\mathbb{T}^+ : \mathcal{H} \longrightarrow \mathcal{H}^+$,
see Theorem~\ref{Trace}. Similarly, we define the trace operator $\mathbb{T}^-$, which associates initial data on $\Sigma_0$ to scattering data on $\mathfrak{H}^- \cup \mathscr{I}^-$.
Using again the energy estimates and Sobolev embeddings, we prove that $\mathbb{T}^+$ is injective and locally Lipschitz. The Goursat problem is shown to be well-posed by applying the results of~\cite{Jo2012,Jo2019} on the domain $\bar{\mathcal{B}}_I \setminus (\mathcal{O} \cup \mathcal{V})$, where $\mathcal{O}$ is a neighborhood of $i^+$ and $\mathcal{V}$ is a neighborhood of $i_0$, both lying away from the support of the initial data on $\mathfrak{H}^+ \cup \mathscr{I}^+$ (see Theorem~\ref{Goursat}). The well-posedness of the Goursat problem implies that $\mathbb{T}^+$ is surjective and that its inverse $(\mathbb{T}^+)^{-1}$ is locally Lipschitz. The same properties hold for $\mathbb{T}^-$.
Therefore, we obtain the conformal scattering operator for nonlinear wave equations
$\mathbb{S} = \mathbb{T}^+ \circ (\mathbb{T}^-)^{-1} : \mathcal{H}^- \longrightarrow \mathcal{H}^+,$
which is a bounded linear operator and locally Lipschitz.

Our paper is organized as follows. In Section~\ref{S2}, we recall the {Schwarzschild} metric, introduce suitable coordinate systems, describe Penrose's conformal compactification, construct a foliation of the future domain $\mathcal{I}^+(\Sigma_0)$, and derive a divergence identity for the nonlinear wave equation. In Section~\ref{S3}, we define the energies on null and spacelike hypersurfaces and establish several energy estimates. In Section~\ref{S4}, we prove the well-posedness of the Cauchy and Goursat problems and construct the conformal scattering operator for nonlinear wave equations.\\ 
{\bf Acknowledgment.} P.T. Xuan was funded by the Postdoctoral Scholarship Programme of Vingroup Innovation Foundation (VINIF), code VINIF.2023.STS.55.

\section{Schwarzschild spacetime and nonlinear wave equations}\label{S2} 

\subsection{Schwarzschild metric, coordinates and foliation}
We consider a four-dimensional Schwarzschild spacetime $\mathcal{M} = \mathbb{R}_t \times (0,+\infty)_r \times S^2_\omega$,
equipped with the Lorentzian metric
\begin{equation}
g = F \,\d t^2 - F^{-1} \,\d r^2 - r^2 \,\d \omega^2,
\end{equation}
where $F := F(r) = 1 - \dfrac{2M}{r}$, $\d \omega^2$ denotes the standard metric on the unit $2$-sphere $S^2$, and $M>0$ is the mass of the black hole. In this paper, we work on the exterior region of the black hole, $
\mathcal{B}_I := \{ r > 2M \}$.

We recall several coordinate systems on Schwarzschild spacetime. First, the Regge--Wheeler variable is defined by
$r_* = r + 2M \log (r - 2M)$,
so that $\d r = F \,\d r_*$. In the coordinates $(t, r_*, \omega)$, the Schwarzschild metric takes the form
\begin{equation}
g = F(\d t^2- \d r_*^2) - r^2\d\omega^2.
\end{equation}
We also introduce the retarded and advanced Eddington--Finkelstein coordinates $u$ and $v$, defined by
\begin{equation}
u=t-r_*, \, v= t+r_*.
\end{equation}
In terms of the coordinates $(u, v, \omega)$, the Schwarzschild metric can be written as
\begin{equation}
g = F\d u \d v - r^2\d\omega^2.
\end{equation}

In Schwarzschild spacetime, there are two families of null geodesics, called the principal null geodesics, which are the integral curves of the outgoing and incoming principal null directions
\begin{equation}
l =\partial_v = \partial_t + \partial_{r_*} \hbox{  and  } n = \partial_u = \partial_t - \partial_{r_*},
\end{equation}
respectively.

We now recall the Penrose conformal compactification of the exterior domain $(\mathcal{B}_I, g)$. We refer the reader to~\cite{Ni1995,Ni2016} for further details. By setting $\Omega = 1/r$ and $\hat{g} = \Omega^2 g$, one obtains a conformal compactification of the exterior region in the coordinates $(u, R = 1/r, \omega)$, namely $
\left( \mathbb{R}_u \times \left[0, \frac{1}{2M}\right] \times S^2_\omega, \, \hat{g} \right)$, equipped with the rescaled metric
\begin{equation}
\hat{g} = R^2(1-2MR) \d u^2 - 2\d u\d R - \d\omega^2.
\end{equation}
The future null infinity $\scri^+$ and the past event horizon $\mathfrak{H}^-$ are null hypersurfaces in the conformally rescaled spacetime
$$\scri^+ = \mathbb{R}_u \times \left\{ 0\right\}_R \times S_\omega^2, \, \mathfrak{H}^- = \mathbb{R}_u \times \left\{ 1/2M\right\}_R \times S^2_\omega.$$

Similarly, in the advanced coordinates $(v, R, \omega)$, the rescaled metric $\hat{g}$ takes the form
\begin{equation}
\hat{g} = R^2(1-2MR)\d v^2 + 2 \d v \d R - \d\omega^2.
\end{equation}
Past null infinity $\scri^-$ and the future event horizon $\mathfrak{H}^+$ are described as null hypersurfaces
$$\scri^- = \mathbb{R}_v \times \left\{ 0\right\}_R \times S_\omega^2, \, \mathfrak{H}^+ = \mathbb{R}_v \times \left\{ 1/2M\right\}_R \times S^2_\omega.$$
The Penrose conformal compactification of $\mathcal{B}_I$ is the spacetime
$$(\bar{\mathcal{B}_I},\hat{g}), \, \bar{\mathcal{B}_I} = \mathcal{B}_I \cup \scri^+ \cup \mathfrak{H}^+ \cup \scri^-\cup \mathfrak{H}^-\cup S_c^2,$$
where $S_c^2$ is the bifurcation sphere (also called the crossing sphere). 
\begin{remark}
Note that there are three singular points of the rescaled metric $\hat{g}$, which correspond to the points missing from the conformal boundary $\scri^+ \cup \mathfrak{H}^+ \cup \scri^- \cup \mathfrak{H}^-$. Specifically, these are the future timelike infinity $i^+$, which is the limit point of uniformly timelike curves as $t \to +\infty$, the past timelike infinity $i^-$, which is the symmetric of $i^+$ in the distant past, and the spacelike infinity $i_0$, which is the limit point of uniformly spacelike curves as $r \to +\infty$. Moreover, these points can be represented as $2$-spheres. For further details on the geometry of Schwarzschild spacetime, we refer the reader to~\cite{Cha}. 
\end{remark}
In order to define a global timelike vector field on $\bar{\mathcal{B}}_I$, we choose an increasing function $\lambda(r) \geq r_*$ such that $2 - F(r)\lambda'(r) > 0$ and $\lambda(r) = r_*$ for $r \geq \dfrac{5M}{2}$. Setting
$\tilde{v} = v - \lambda(r) = t + r_* - \lambda(r)$,
we have $\tilde{v} = t$ when $r \geq \dfrac{5M}{2}$. In the coordinates $(\tilde{v}, r, \omega)$, the Schwarzschild metric takes the form
\begin{equation}
g= F\d \tilde{v}^2 - 2( 1 - F\lambda'(r))\d \tilde{v}\d r - (2\lambda'(r) - F(\lambda'(r))^2)\d r^2 - r^2\d \omega^2.
\end{equation}
{Observe that the vector field \( K = \partial_{\tilde{v}} \) is Killing and transverse to the horizon. Moreover, one can check that the hypersurface $\widetilde{\Sigma}_\tau = \{ \tilde{v} = \text{constant} \}$
is spacelike, since the gradient \( \nabla \tilde{v} \) is timelike (see~\cite{Yang}; for further details, see~\cite{Mar}).}

{Let $r_{FH} > 5M$ be a sufficiently large constant.} We define the following hypersurface
$\mathcal{S}_\tau = \mathcal{N}_\tau \cup \mathcal{H}_\tau$, where
\begin{equation}
\mathcal{N}_{\tau} = \widetilde{\Sigma}_{\tau+(r_{FH})_*} \cap \left\{  2M \leq r < r_{FH}\right\}, {(r_{FH})_* = r_{FH} + 2M\log(r_{FH}-2M)} 
\end{equation}
and
\begin{equation}
\mathcal{H}_\tau = \left\{u = \tau,\, \tilde{v}\geq 0  \right\} \cap \{ r\geq r_{FH} \}.
\end{equation}
The future $\mathcal{I}^+(\mathcal{S}_0)$ can be foliated by $\{\mathcal{S}_\tau \}_{\tau\geq 0}$. 

The volume forms {associated with} the Schwarzschild metric $g$ and the rescaled metric $\hat{g}$ are given as follows
\begin{equation}
\mathrm{dVol}_g = r^2F \d t \wedge \d r_* \wedge \d^2\omega = \frac{r^2F}{2}\d u \wedge\d v \wedge \d^2\omega = r^2F \d \tilde{v} \d r_*  \d^2\omega.
\end{equation}
and 
\begin{equation}
\mathrm{dVol}_{\hat g} = R^2F \d t \wedge \d r_* \wedge \d^2\omega = \frac{R^2F}{2} \d u \wedge \d v \wedge \d^2\omega = R^2F \d \tilde{v}\d r_* \d^2\omega,
\end{equation}
respectively, where $\d^2 \omega$ denotes the standard area element on the unit $2$-sphere $S^2$.

\begin{figure}[H]
\begin{center}
\includegraphics[scale=0.6]{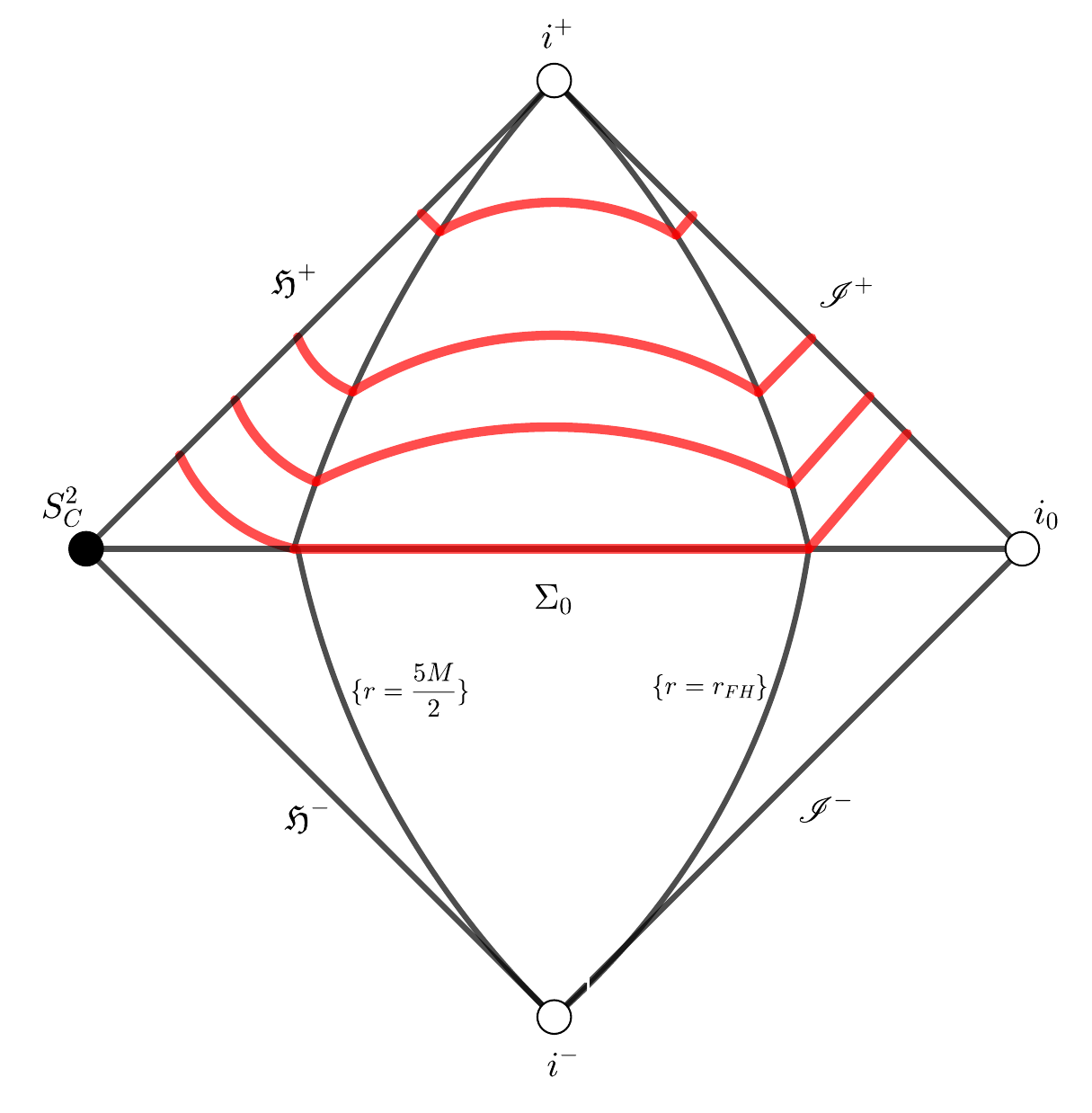}
\caption{Conformal compactification diagram and foliation $\left\{ \mathcal{S}_\tau \right\}_\tau$ of $\mathcal{I}^+(\mathcal{S}_0)$}
\end{center}
\end{figure}

\subsection{The nonlinear wave equation and  conservation law}\label{S23}
We consider the defocusing semilinear wave equation on $(\mathcal{B}_I,g)$:
\begin{equation}\label{nonlinearequation}
\Box_g\psi + |\psi|^2\psi =0,
\end{equation}
where $\Box_g$ is the scalar Laplace--Beltrami operator {associated with} the Schwarzschild metric $g$, given by
$$\Box_g = \frac{1}{F}\left( \frac{\partial^2}{\partial t^2} - \frac{1}{r^2}\frac{\partial}{\partial r_*} r^2 \frac{\partial}{\partial r_*} \right) - \frac{1}{r^2}\Delta_{S^2}.$$
{Note that Equation \eqref{nonlinearequation} is a special case of the conformally coupled wave equation in four dimensions,
\begin{equation}\label{ConForEqs}
\square_g\psi + |\psi|^2\psi + \frac{1}{6} S_g \, \psi = 0,
\end{equation}
where $S_g$ denotes the scalar curvature, which is identically zero for the Schwarzschild metric $g$.}

Setting $\hat{\psi} = \Omega^{-1} \psi = r \psi$, and using \eqref{ConForEqs}, we obtain the rescaled equation, that is, the conformal form of \eqref{nonlinearequation} on $ (\bar{\mathcal{B}}_I, \hat{g})$, as follows (see~\cite{Pha2020,NiXu2019} and, for further details, see~\cite[Vol.~1]{PeRi84}):
\begin{equation}\label{rescaledequation}
\Box_{\hat g}\hat{\psi} + |\hat\psi|^2\hat{\psi} +  2MR\hat{\psi}  =0,
\end{equation}
{where $\dfrac{1}{6}S_{\hat{g}} = R^{-3}\Box_g R = 2MR$} and $\Box_{\hat g}$ is the scalar Laplace--Beltrami operator {associated with} the conformal metric $\hat{g}$, given by
$$\Box_{\hat g} = \frac{r^2}{F}\left( \frac{\partial^2}{\partial t^2} - \frac{\partial}{\partial r_*} \right) - \Delta_{S^2}.$$

We consider the following stress--energy tensor {associated with} the linear wave equation $\Box_{\hat g}\hat{\psi} = 0$:
$$\hat{T}_{ab} = \hat{\nabla}_a\hat\psi \hat{\nabla}_b\hat\psi - \frac{1}{2}\left<\hat{\nabla}\hat\psi,\hat{\nabla}\hat\psi\right>_{\hat g} \hat{g}_{ab}.$$
The divergence of $\hat{T}_{ab}$ is
$$\hat{\nabla}^a\hat{T}_{ab} = \left( \Box_{\hat g}\hat{\psi} \right)\hat{\nabla}_b\hat{\psi} = (- 2MR\hat{\psi} - |\hat\psi|^2\hat{\psi}) \hat{\nabla}_b(\hat\psi).$$
By contracting $\hat{T}_{ab}$ with the timelike vector field $\partial_t$, we obtain the energy current
$\hat{J}_a = (\partial_t)^b\hat{T}_{ab}$.

Setting $V = \left(MR\hat{\psi}^2 + \dfrac{1}{4}\psi^4 \right)\partial_t$ and using the fact that $\hat{\Gamma}^{\mathbf{a}}_{\mathbf{a}0} = 0$ in the coordinates $(t, r, \omega)$, we have
$$\mathrm{div}V= \hat{\nabla}_aV^a = \frac{\partial}{\partial t}\left( MR \hat{\psi}^2 + \frac{1}{4}\hat\psi^4 \right) + \hat{\Gamma}^{\bf a}_{\bf {a} 0}V^0 = (2MR\hat{\psi} + |\hat\psi|^2\hat{\psi})\partial_t\hat\psi.$$
This yields the following divergence equation
\begin{equation}\label{conser}
\hat{\nabla}_a (\hat{J}^a + V^a) = 0, \hbox{ with } V= \left( MR\hat{\psi}^2 + \frac{1}{4}\hat{\psi}^4\right)\partial_t.
\end{equation}

\section{Energy fluxes, energy decay and pointwise decay}\label{S3}

\subsection{Energy fluxes through spacelike and null hypersurfaces}
We recall the Hodge dual of a one-form $\alpha$ on a spacetime $({\cal M}, g)$ (see \cite{Pha2020} and, for more details, \cite{PeRi84}):
\begin{equation*}
(*\alpha)_{abcd} = e_{abcd}{\alpha}^d,
\end{equation*}
where $e_{abcd}$ is the volume form on $({\cal M}, g)$, denoted simply by $\mathrm{dVol}_g$. The differential form of $*\alpha$ is
\begin{equation*}
\d *\alpha = -\frac{1}{4}(\nabla_a\alpha^a)\mathrm{{dVol}}_g.
\end{equation*}
Hence, if the surface ${S}$ is the boundary of a bounded open set $\Omega$ and has an outgoing orientation, then, as a consequence of Stokes' theorem, we have
\begin{equation}\label{Stokesformula}
-4\int_{{S}}*\alpha = \int_{\Omega}(\nabla_a\alpha^a)\mathrm{{dVol}}_g.
\end{equation}

Assume that $\hat{\psi}$ is a solution to the Cauchy problem for equation \eqref{rescaledequation} with smooth and compactly supported initial data. We define the rescaled energy flux $\hat{\mathbb{E}}^{\partial_t}_{\mathcal S}$ through an oriented hypersurface $S$ by
\begin{equation}\label{EN}
\hat{\mathbb{E}}^{\partial_t}_S[\hat\psi; 3] = -4\int_{\mathcal S} *(\hat{J}_a + V_a)\d x^a = \int_{\mathcal S} (\hat{J}_a+V_a)\hat{N}^a\hat{L}\hook \mathrm{{dVol}}_{\hat g},
\end{equation}
where $\hat{L}$ is transverse to $\mathcal{S}$ and $\hat{N}$ denotes the normal vector field to $\mathcal{S}$, satisfying
$\hat{L}^a \hat{N}_a = 1$.

We now use formula \eqref{EN} to determine precisely the energy fluxes of $\hat{\psi}$ through the hypersurfaces $\Sigma_0$, $\mathcal{S}_T$, and $\mathcal{H}_T$. 
First, on $\Sigma_0$ we choose the vector fields
$\hat{L} = \dfrac{r^2}{F}\partial_t$ and $\hat{N} = \partial_t$.
Next, on future null infinity $\scri^+$, we choose $\hat{L}_{\scri^+} = -\partial_R$ in the $(u, R, \omega)$ coordinates. 
Hence, we have
$\hat{L}_{\scri^+} = \dfrac{r^2 F^{-1}}{2}\, l \big|_{\scri^+}$.
Moreover, on the null hypersurface $\mathfrak{H}^+$, we choose $\hat{L}_{\mathfrak{H}^+} = \partial_R$ in the $(v, R, \omega)$ coordinates. 
This leads to
$\hat{L}_{\mathfrak{H}^+} = \dfrac{r^2 F^{-1}}{2}\, n \big|_{\mathfrak{H}^+}$.

Clearly, we have $\hat{N} = \partial_t$ on both $\mathfrak{H}^+$ and $\scri^+$. 
This corresponds to $\hat{N} = \partial_v$ on $\mathfrak{H}^+$ and $\hat{N} = \partial_u$ on $\scri^+$. 
On the other hand, the vector field $V$ is parallel to $\partial_t$, which is null on $\mathfrak{H}^+$ and $\scri^+$. 
Therefore, in these two cases we have $V_a \hat{N}^a = 0$. Consequently, by using formula \eqref{EN}, we can easily compute the energy fluxes through $\Sigma_0$, $\scri^+_T$, and $\mathfrak{H}^+_T$ as
\begin{eqnarray}
\hat{\mathbb{E}}^{\partial_t}_{\Sigma_0}[\hat\psi;3] &=&  \int_{\Sigma_0}(\hat{J}_a+V_a)(\partial_t)^a r^2F^{-1}\partial_t \hook \mathrm{dVol}_{\hat g} \cr
&=& \frac{1}{2}\int_{\Sigma_0} \left( (\partial_t\hat{\psi})^2 + (\partial_{r_*}\hat{\psi})^2 \right.\cr
&&\left.\hspace{3cm} + R^2F|\nabla_{S^2}\hat{\psi}|^2 + (2MR\hat{\psi}^2 + \frac{1}{2}\hat{\psi}^4)FR^2 \right) \d r_*\d^2\omega,\cr
\hat{\mathbb{E}}^{\partial_t}_{\mathfrak{H}^+_T}[\hat\psi;3] &=&  \int_{\mathfrak{H}_T^+}\hat{J}_a(\partial_v)^a\hat{L}_{\mathfrak{H}^+}\hook \mathrm{dVol}_{\hat g} = \int_{\mathfrak{H}^+} (\partial_v\hat{\psi})^2 \d v\d^2\omega,\cr
\hat{\mathbb{E}}^{\partial_t}_{\scri^+_T}[\hat\psi;3] &=&  \int_{\scri_T^+}\hat{J}_a(\partial_u)^a\hat{L}_{\scri^+}\hook \mathrm{dVol}_{\hat g} = \int_{\scri_T^+} (\partial_u\hat{\psi})^2  \d u\d^2\omega.
\end{eqnarray}
Since the energy fluxes through the null hypersurfaces $\mathfrak{H}_T^+$ and $\scri_T^+$ do not contain the term $|\hat{\psi}|^4$, we may simply write
$$\hat{\mathbb{E}}^{\partial_t}_{\mathfrak{H}^+_T}[\hat\psi] = \hat{\mathbb{E}}^{\partial_t}_{\mathfrak{H}^+_T}[\hat\psi;3] \hbox{   and   } \hat{\mathbb{E}}^{\partial_t}_{\scri^+_T}[\hat\psi] = \hat{\mathbb{E}}^{\partial_t}_{\scri^+_T}[\hat\psi;3].$$

Moreover, the energy flux through $\mathcal{S}_T = \widetilde{\Sigma}_{T + (r_{FH})_*} \cup \mathcal{H}_T$ can be computed explicitly as follows:
\begin{eqnarray}\label{resenergy}
\hat{\mathbb{E}}^{\partial_t}_{\mathcal{S}_T}[\hat\psi;3] &=& \hat{\mathbb{E}}^{\partial_t}_{\widetilde{\Sigma}_{T+(r_{FH})_*}}[\hat\psi;3] + \hat{\mathbb{E}}^{\partial_t}_{\mathcal{H}_T}[\hat\psi;3]\cr
&=& \frac{1}{2}\int_{\widetilde{\Sigma}_{T+(r_{FH})_*}\cap \{ 2M\leq r< \frac{5M}{2} \}} \left( (\partial_{\tilde v}\hat{\psi})^2  +  R^2F |\nabla_{S^2}\hat\psi|^2 \right.\cr
&&\left.\hspace{7cm} + (2MR\hat{\psi}^2 + \frac{1}{2} \hat{\psi}^4)FR^2 \right) \d \tilde{v} \d^2\omega\cr
&&+\frac{1}{2}\int_{\widetilde{\Sigma}_{T+(r_{FH})_*}\cap \{ \frac{5M}{2}\leq r < r_{FH} \}} \left( (\partial_t\hat{\psi})^2 + (\partial_{r_*}\hat{\psi})^2 + R^2F|\nabla_{S^2}\hat{\psi}|^2 \right.\cr
&&\left.\hspace{7cm} + (2MR\hat{\psi}^2 + \frac{1}{2}\hat{\psi}^4) FR^2 \right) \d r_*\d^2\omega\cr
&&+ \int_{\mathcal{H}_T} \left( (\partial_v\hat{\psi})^2  +  R^2F |\nabla_{S^2}\hat\psi|^2 + R^2F\hat{\psi}^2 \right) \d v\d^2\omega.
\end{eqnarray}
\begin{remark}
We note that the higher-order term $|\hat{\psi}|^4$ appears only in the energy formula on the spacelike hypersurface $\Sigma_0$ and on the spacelike part of the hypersurface $\mathcal{S}_T$. 
This term can be controlled by the remaining terms in the energy formula using the Sobolev embedding $H^1 \hookrightarrow L^6$ on spacelike hypersurfaces (see Lemma~\ref{Sobolev} below).
\end{remark}

\subsection{Pointwise and energy decays and two sides of energy estimates}
By straightforward calculations, we obtain the following transformation from the integral of $(\partial_{r_*}\psi)^2$ on the spacelike hypersurface $\Sigma_T = \{ t = T \}$ to that of $(\partial_{r_*}\hat{\psi})^2$:
\begin{eqnarray*}
\int_{\Sigma_T}(\partial_{r_*}\psi)^2 r^2\d r_* \d^2\omega &=&\int_{\Sigma_T}(\partial_{r_*}(R\hat{\psi}))^2 r^2\d r_* \d^2\omega = \int_{\Sigma_T}\left(\frac{r\partial_{r_*}\hat{\psi}-\hat{\psi}\partial_{r_*}r}{r^2}\right)^2 r^2 \d r_* \d^2\omega\cr
&=&\int_{\Sigma_T}\left(\partial_{r_*}\hat{\psi}-\frac{F}{r}\hat{\psi}\right)^2 \d r_* \d^2\omega\cr
&=&\int_{\Sigma_T}\left((\partial_{r_*}\hat{\psi})^2 - 2(\partial_{r_*}\hat{\psi})\frac{F}{r}\hat{\psi} + \frac{F^2}{r^2}\hat{\psi}^2 \right) \d r_* \d^2\omega\cr
&=&\int_{\Sigma_T}\left((\partial_{r_*}\hat{\psi})^2 + \frac{FF'}{r}\hat{\psi}^2 - \partial_{r_*}\left( \frac{F}{r}\hat{\psi}^2 \right) \right) \d r_* \d^2\omega\cr
&=& \int_{\Sigma_T}\left((\partial_{r_*}\hat{\psi})^2 + 2MFR^3\hat{\psi}^2 \right) \d r_* \d^2\omega,
\end{eqnarray*}
where we used the facts that $F'= d( 1 - 2MR )/dr =2MR^2$ and $\int_{\Sigma_T}\partial_{r_*}\left( \frac{F}{r}\hat{\psi}^2 \right) \d r_* \d^2\omega = 0$ for the smooth, compactly supported function $\hat{\psi}$. Consequently, we obtain the following expression for the energy flux through $\Sigma_0$ in terms of the original field:
\begin{equation}\label{EnEq}
\hat{\mathbb{E}}^{\partial_t}_{\Sigma_0}[\hat{\psi};3] =\frac{1}{2}\int_{\Sigma_0}\left(r^2(\partial_t{\psi})^2 + r^2 (\partial_{r_*}{\psi})^2 + F|\nabla_{S^2}\psi|^2 + \frac{F}{2}{\psi}^4 \right) \d r_* \d^2\omega := \mathbb{E}_{\Sigma_0}^{\partial_t}[\psi;3].
\end{equation}

Applying Assertion~$(iii)$ of Theorem~1 in \cite{Yang}, we derive the following pointwise decay for the solution $\psi$ to \eqref{nonlinearequation}: for $1 < \gamma < 2$, there exists a positive constant $C = C(\gamma, M, \mathcal{E}_0)$ such that
\begin{equation}\label{pointwise}
|\psi|\leq \begin{cases}
    C\mathcal{E}_1^C\sqrt{\mathcal{E}_2}\tilde{v}^{-\frac{\gamma}{2}} & \hbox{   if  } 2M\leq r \leq 5M, \cr
    C\mathcal{E}_1^C\sqrt{\mathcal{E}_2}(rv)^{-\frac{1}{2}}(1+|u|)^{-\frac{\gamma-1}{2}} & \hbox{   if  } r>5M,
 \end{cases}
\end{equation}
where we define
\begin{equation}
\mathcal{E}_0 = \mathcal{E}_{\widetilde{\Sigma}_0}[\psi; 3], 
\qquad 
\mathcal{E}_k = \mathcal{E}_{\widetilde{\Sigma}_0}\!\left(Z^{\leq k}\psi\right),
\end{equation}
with $Z^k = Z_1 Z_2 \cdots Z_k$, where each $Z_j$ is a Killing vector field in the set
\[
\Gamma = \left\{ \partial_{\tilde v}, \ \Omega_{ij} = x_i \partial_j - x_j \partial_i \right\},
\]
and $x = r \omega = (x_1, x_2, x_3)$ in the coordinates $(\tilde{v}, r, \omega)$. 
In addition, the energy $\mathcal{E}_{\widetilde{\Sigma}_0}$ is given by (see \cite{Yang} for details):
$$\mathcal{E}_{\widetilde{\Sigma}_0} = \mathbb{E}_{\widetilde{\Sigma}_0} + \frac{1}{2}\int_{\widetilde{\Sigma}_0}F\psi^2\d r_* \d^2\omega.$$
However, the second term can in fact be controlled by the first one. Therefore, in this paper it is sufficient to use \eqref{EnEq}.

Furthermore, applying the energy flux decay results in Propositions~4.1 and~5.1 of \cite{Yang}, we find that for $1 < \gamma_0, \gamma < 2$ and $N > 2$,
\begin{equation}\label{EnergyDecay}
\mathcal{E}^{\partial_t}_{\mathcal{S}_\tau}[\psi;3] \lesssim (4+\tau^2)^{-\frac{\gamma_0}{2}} + (4+\tau^2)^{-\frac{N}{4}-\frac{\gamma + 1}{2}}\mathcal{E}_1^C\mathcal{E}_2.
\end{equation}
This yields
\begin{equation}\label{lim1}
\lim_{\tau \to +\infty}\mathcal{E}^{\partial_t}_{\mathcal{S}_\tau}[\psi;3] = 0.
\end{equation}

Combining \eqref{resenergy} and \eqref{EnEq} with the pointwise decay \eqref{pointwise}, the energy decay \eqref{EnergyDecay}, and the limit \eqref{lim1}, we obtain
\begin{eqnarray*}
\hat{\mathbb{E}}^{\partial_t}_{\mathcal{S}_T}[\hat\psi;3] &=& \hat{\mathbb{E}}^{\partial_t}_{\widetilde{\Sigma}_{T+(r_{FH})_*}}[\hat\psi;3] + \hat{\mathbb{E}}^{\partial_t}_{\mathcal{H}_T}[\hat\psi;3]\cr
&=& \int_{\widetilde{\Sigma}_{T+(r_{FH})_*}\cap \{ 2M\leq r< \frac{5M}{2} \}} \left( (\partial_{\tilde v}\hat{\psi})^2  +  R^2F |\nabla_{S^2}\hat\psi|^2 + \left( 2MR\hat{\psi}^2 + \hat{\psi}^4 \right)FR^2 \right) \d \tilde{v} \d^2\omega\cr
&&+ \hat{\mathbb{E}}^{\partial_t}_{{\Sigma}_{T+(r_{FH})_*}\cap \{ \frac{5M}{2}\leq r < r_{FH} \}}[\hat{\psi};3] + \int_{\mathcal{H}_T} \left( (\partial_v\hat{\psi})^2  +  R^2F |\nabla_{S^2}\hat\psi|^2 + R^2F\hat{\psi}^2 \right) \d v\d^2\omega\cr
&=& {\mathbb{E}}^{\partial_t}_{{\Sigma}_{T+(r_{FH})_*}\cap \{ \frac{5M}{2}\leq r < r_{FH} \}}[{\psi};3]\cr 
&&+ \int_{\widetilde{\Sigma}_{T+(r_{FH})_*}\cap \{ 2M\leq r< \frac{5M}{2} \}} \left( (\partial_{\tilde v}\hat{\psi})^2  +  R^2F |\nabla_{S^2}\hat\psi|^2 + \left( 2MR\hat{\psi}^2 + \hat{\psi}^4 \right)FR^2 \right) \d \tilde{v} \d^2\omega\cr
&&+ \int_{\mathcal{H}_T} \left( (\partial_v\hat{\psi})^2  +  R^2F |\nabla_{S^2}\hat\psi|^2 + R^2F\hat{\psi}^2 \right) \d v\d^2\omega\cr
&\leq& {\mathbb{E}}^{\partial_t}_{{\Sigma}_{T+(r_{FH})_*}\cap \{ \frac{5M}{2}\leq r < r_{FH} \}}[{\psi};3] + {\mathbb{E}}^{\partial_t}_{{\Sigma}_{T+(r_{FH})_*}\cap \{ 2M\leq r < \frac{5M}{2} \}}[{\psi};3] + {\mathbb{E}}^{\partial_t}_{\mathcal{H}_T}[\psi;3] \cr 
&&+ \int_{\widetilde{\Sigma}_{T+(r_{FH})_*}\cap \{ 2M\leq r< \frac{5M}{2} \}} (F^2+2MR)|\psi|^2 \d u \d^2\omega + \int_{\mathcal{H}_T} (F^2+ F)|\psi|^2 \d v\d^2\omega\cr
&\leq& \mathbb{E}^{\partial_t}_{\mathcal{S}_T}[\psi;3] + \int_{\widetilde{\Sigma}_{T+(r_{FH})_*}\cap \{ 2M\leq r< \frac{5M}{2} \}} (F^2+2MR)|\psi|^2 \d u \d^2\omega + \int_{\mathcal{H}_T} (F^2+ F)|\psi|^2 \d v\d^2\omega\cr
&\leq& \mathbb{E}^{\partial_t}_{\mathcal{S}_T}[\psi;3] + \int_{\widetilde{\Sigma}_{T+(r_{FH})_*}\cap \{ 2M\leq r< \frac{5M}{2} \}} (F^2+2MR) C^2\mathcal{E}_1^{2C}\mathcal{E}_2\tilde{v}^{-\gamma} \d u \d^2\omega\cr
&&+ \int_{\mathcal{H}_T} (F^2+ F)C^2\mathcal{E}_1^{2C}\mathcal{E}_2 (rv)^{-1}(1+|u|)^{-(\gamma-1)} \d v\d^2\omega\cr
&\lesssim& \mathbb{E}^{\partial_t}_{\mathcal{S}_T}[\psi;3] + 2\pi \int^{+\infty}_{T-(\frac{5M}{2})_*} (u+2r_*-\lambda(r))^{-\gamma}  \d u + 2\pi \int_{T+(r_{FH})_*}^{+\infty}(v-u|_{\mathcal{H}_T})^{-1} v^{-1}\d v\cr
&\lesssim& \mathbb{E}^{\partial_t}_{\mathcal{S}_T}[\psi;3] + 2\pi \int^{+\infty}_{T-(\frac{5M}{2})_*} \left(u+2(2M)_*- (\frac{5M}{2})_*\right)^{-\gamma}  \d u + 2\pi \int_{T+(r_{FH})_*}^{+\infty}(v-u|_{\mathcal{H}_T})^{-2}\d v\cr
&&\hbox{(for $T$ large enough)}\cr
&\lesssim& \mathcal{E}^{\partial_t}_{\mathcal{S}_T}[\psi;3] + \frac{2\pi}{\gamma-1} \left(T-(\frac{5M}{2})_*+2(2M)_*- (\frac{5M}{2})_*\right)^{1-\gamma} + 2\pi (T+(r_{FH})_*-u|_{\mathcal{H}_T})^{-1}\cr
&\longrightarrow& 0
\end{eqnarray*}
as $t\to +\infty$. Consequently, we have
\begin{equation}\label{limit2}
\lim_{T\to +\infty} \hat{\mathbb{E}}^{\partial_t}_{\mathcal{S}_T}[\hat{\psi};3] =0.
\end{equation}

With the aid of the limit \eqref{limit2}, we establish the energy identity up to $i^+$ in the following theorem.
\begin{theorem}\label{equality1}
Consider the Cauchy problem for the rescaled nonlinear wave equation \eqref{rescaledequation} with smooth and compactly supported initial data on $\Sigma_0 = \{ t = 0 \}$. 
We can define the energy fluxes of the rescaled solution $\hat{\psi}$ through the null conformal boundary $\mathfrak{H}^+ \cup \scri^+$ by
$$\hat{\mathbb{E}}^{\partial_t}_{\scri^+}[\hat\psi] + \hat{\mathbb{E}}^{\partial_t}_{\mathfrak{H}^+}[\hat\psi] = \lim_{T\rightarrow\infty} \left( \hat{\mathbb{E}}^{\partial_t}_{\scri_T^+}[\hat\psi] + \hat{\mathbb{E}}^{\partial_t}_{\mathfrak{H}_T^+}[\hat\psi] \right).$$
Furthermore, we have the following equality:
$$\hat{\mathbb{E}}^{\partial_t}_{\scri^+}[\hat\psi] + \hat{\mathbb{E}}^{\partial_t}_{\mathfrak{H}^+}[\hat\psi] = \hat{\mathbb{E}}^{\partial_t}_{\Sigma_0}[\hat{\psi};3].$$
\end{theorem} 
\begin{proof} 
By integrating the divergence equation \eqref{conser} over the domain $\Omega$, which is closed by the Cauchy initial hypersurface $\Sigma_0$ and the hypersurfaces $\mathfrak{H}_T^+$, $\mathcal{S}_T$, and $\scri_T^+$ (for $T > 0$), and applying \eqref{Stokesformula}, we have
\begin{equation}\label{energyidentity}
\hat{\mathbb{E}}^{\partial_t}_{\Sigma_0}[\hat\psi;3] = \hat{\mathbb{E}}^{\partial_t}_{\scri_T^+}[\hat\psi;3] + \hat{\mathbb{E}}^{\partial_t}_{\mathfrak{H}_T^+}[\hat\psi;3] + \hat{\mathbb{E}}^{\partial_t}_{\mathcal{S}_T}[\hat\psi;3].
\end{equation}
Hence, the total energy flux across $\scri_T^+$ and $\mathfrak{H}_T^+$ is a nonnegative and increasing function of $T$. 
In view of \eqref{energyidentity} and the positivity of $\hat{\mathbb{E}}^{\partial_t}_{\mathcal{S}_T}[\hat{\psi}]$, it is bounded above by $\hat{\mathcal{E}}_{\Sigma_0}(\hat{\psi})$. 
Consequently, the limit of
$\hat{\mathbb{E}}^{\partial_t}_{\mathfrak{H}_T^+}(\hat{\psi}) + \hat{\mathbb{E}}^{\partial_t}_{\scri_T^+}(\hat{\psi})$
exists, and the following sum is well defined:
\begin{eqnarray}\label{limitenergy}
\hat{\mathbb{E}}^{\partial_t}_{\scri^+}[\hat\psi] + \hat{\mathbb{E}}^{\partial_t}_{\mathfrak{H}^+}[\hat\psi] &=& \lim_{T\rightarrow\infty}\left( \hat{\mathbb{E}}^{\partial_t}_{\scri_T^+}[\hat\psi;3] + \hat{\mathbb{E}}^{\partial_t}_{\mathfrak{H}_T^+}[\hat\psi]\right)\cr
&=& \hat{\mathbb{E}}^{\partial_t}_{\Sigma_0}[\hat\psi;3] - \lim_{T\rightarrow\infty}\hat{\mathbb{E}}^{\partial_t}_{\mathcal{S}_T}[\hat\psi;3]\cr
&=& \hat{\mathbb{E}}^{\partial_t}_{\Sigma_0}[\hat\psi;3]
\end{eqnarray}
The last equality follows from \eqref{limit2}.
\end{proof}
We now recall the Sobolev embedding on spacelike hypersurfaces in the following lemma (see \cite{Jo2012,NiXu2019} for details):
\begin{lemma}\label{Sobolev}
Consider the infinite half-cylinder $({\mathscr{H}_+} = (0,+\infty)\times S^2_\omega,h)$, where the metric $h$ is given by $h= \d x^2 + \d \omega^2$. Then, the embedding ${H}^1(\Sigma_0) \hookrightarrow L^6(\Sigma_0)$ holds, i.e.,
\begin{equation}
\left( \int_{{\mathscr{H}_+}} |\phi|^6 \d \mathrm{Vol}_{{\mathscr{H}_+}} \right)^{\frac{1}{3}} \leq C \int_{{\mathscr{H}_+}} \left( |\partial_x\phi|^2 + |\nabla_{S^2}\phi|^2 + |\phi|^2 \right) \d \mathrm{Vol}_{{\mathscr{H}_+}}.
\end{equation}
\end{lemma}
\begin{proof}
The proof was given in \cite[Lemma 4.2]{NiXu2019}.
\end{proof}
Clearly, we have
$|\hat{\psi}|^4 \le \frac{1}{2}\left( |\hat{\psi}|^6 + |\hat{\psi}|^2 \right)$.
Hence, the Sobolev embedding established in Lemma~\ref{Sobolev} allows us to control the term $|\hat{\psi}|^4$ appearing in the energy formula on spacelike hypersurfaces.
By using the Sobolev embedding $H^1(\Sigma_0) \hookrightarrow L^6(\Sigma_0)$ and Theorem~\ref{equality1}, we obtain the following corollary.
\begin{corollary}\label{1}
Setting
\begin{equation}\label{normEn}
\hat{\mathbb{E}}^{\partial_t}_{\Sigma_0}[\hat{\psi}] =\frac{1}{2}\int_{\Sigma_0}\left((\partial_t{\hat\psi})^2 + r^2 (\partial_{r_*}{\hat\psi})^2 + R^2F|\nabla_{S^2}\hat\psi|^2 \right) \d r_* \d^2\omega,
\end{equation}
we have
\begin{equation}\label{ineq1}
\hat{\mathbb{E}}^{\partial_t}_{\scri^+}[\hat\psi] + \hat{\mathbb{E}}^{\partial_t}_{\mathfrak{H}^+}[\hat\psi] \leq \left(\left(\hat{\mathbb{E}}^{\partial_t}_{\Sigma_0}[\hat{\psi}]\right)^2 +1\right)\hat{\mathbb{E}}^{\partial_t}_{\Sigma_0}[\hat{\psi}].
\end{equation}
The reverse inequality of \eqref{ineq1} is clearly valid in the following sense:
\begin{equation}\label{ineq2}
\hat{\mathbb{E}}^{\partial_t}_{\Sigma_0}[\hat{\psi}] \leq \hat{\mathbb{E}}^{\partial_t}_{\scri^+}[\hat\psi] + \hat{\mathbb{E}}^{\partial_t}_{\mathfrak{H}^+}[\hat\psi] .
\end{equation}
\end{corollary}
To establish the well-posedness of equation \eqref{rescaledequation}, we require energy estimates for the field on the Cauchy hypersurfaces $\Sigma_0$ and $\Sigma_t$.
\begin{theorem}\label{equality2}
Consider the Cauchy problem for the rescaled nonlinear wave equation \eqref{rescaledequation} with smooth and compactly supported initial data on $\Sigma_0 = \{ t = 0 \}$. 
Then the following energy estimates hold:
\begin{equation}\label{ineq3}
\hat{\mathbb{E}}^{\partial_t}_{\Sigma_0}[\hat{\psi}] \leq \left(\left(\hat{\mathbb{E}}^{\partial_t}_{\Sigma_0}[\hat{\psi}]\right)^4 + 2\left(\hat{\mathbb{E}}^{\partial_t}_{\Sigma_0}[\hat{\psi}]\right)^2 + 2 \right) \hat{\mathbb{E}}^{\partial_t}_{\Sigma_t}[\hat{\psi}]
\end{equation}
and
\begin{equation}\label{ineq4}
\hat{\mathbb{E}}^{\partial_t}_{\Sigma_t}[\hat{\psi}] \leq \left(\left(\hat{\mathbb{E}}^{\partial_t}_{\Sigma_0}[\hat{\psi}]\right)^2+1\right) \hat{\mathbb{E}}^{\partial_t}_{\Sigma_0}[\hat{\psi}].
\end{equation}
\end{theorem} 
\begin{proof}
By integrating the divergence equation \eqref{conser} over the domain $\widetilde{\Omega}$, which is closed by the Cauchy hypersurfaces $\Sigma_0$ and $\Sigma_t$, and applying \eqref{Stokesformula}, we have
\begin{equation}
\hat{\mathbb{E}}^{\partial_t}_{\Sigma_0}[\hat{\psi};3] = \hat{\mathbb{E}}^{\partial_t}_{\Sigma_t}[\hat{\psi};3].
\end{equation}
From the Sobolev embedding $H^1(\Sigma_t) \hookrightarrow L^6(\Sigma_t)$, it follows that 
\begin{eqnarray}
\hat{\mathbb{E}}^{\partial_t}_{\Sigma_0}[\hat{\psi}] &\leq& \hat{\mathbb{E}}^{\partial_t}_{\Sigma_t}[\hat{\psi};3] \leq \left(\left(\hat{\mathbb{E}}^{\partial_t}_{\Sigma_t}[\hat{\psi}]\right)^2+1\right) \hat{\mathbb{E}}^{\partial_t}_{\Sigma_t}[\hat{\psi}]\cr
&\leq& \left(\left(\hat{\mathbb{E}}^{\partial_t}_{\Sigma_0}[\hat{\psi};3]\right)^2+1\right) \hat{\mathbb{E}}^{\partial_t}_{\Sigma_t}[\hat{\psi}]\cr
&\leq& \left(\left(\hat{\mathbb{E}}^{\partial_t}_{\Sigma_0}[\hat{\psi}]\right)^4 + 2\left(\hat{\mathbb{E}}^{\partial_t}_{\Sigma_0}[\hat{\psi}]\right)^2 + 2 \right) \hat{\mathbb{E}}^{\partial_t}_{\Sigma_t}[\hat{\psi}].
\end{eqnarray}
The reverse inequality is obtained once more by applying the Sobolev embedding $H^1(\Sigma_0) \hookrightarrow L^6(\Sigma_0)$:
\begin{equation}
\hat{\mathbb{E}}^{\partial_t}_{\Sigma_t}[\hat{\psi}] \leq \left(\left(\hat{\mathbb{E}}^{\partial_t}_{\Sigma_0}[\hat{\psi}]\right)^2+1\right) \hat{\mathbb{E}}^{\partial_t}_{\Sigma_0}[\hat{\psi}].
\end{equation}

\end{proof}

\section{Cauchy and Goursat problems and scattering operator}\label{S4}
\subsection{Cauchy problem and an injective trace operator}
\begin{definition}
We denote by $\mathcal{H}(\Sigma_0)$ the completion of $\mathcal{C}_0^\infty(\Sigma_0)\times \mathcal{C}_0^\infty(\Sigma_0)$ with respect to the norm
\begin{eqnarray}
\left\|(\hat{\psi}_0,\hat{\psi}_1) \right\|_{\mathcal H(\Sigma_0)} &=& \frac{1}{\sqrt 2}\left(\int_{\Sigma_0} \left( (\hat{\psi}_1)^2 + (\partial_{r_*}\hat{\psi}_0)^2 + R^2F|\nabla_{S^2}\hat{\psi}_0|^2 \right.\right.\cr
&&\left.\left.\hspace{5cm} + (2MR+1)R^2F\hat{\psi}_0^2 \right) \d r_*\d^2\omega \right)^{1/2}.
\end{eqnarray}
An analogous definition holds for the space $\mathcal{H}(\Sigma_\tau)$ on the hypersurface $\Sigma_\tau = \{ t = \tau \}$.
\end{definition}

By extending the methods of \cite{CaCho, Pha2023} and using the energy estimates \eqref{ineq3}, \eqref{ineq4} and the energy decay property \eqref{limitenergy}, we establish the following well-posedness result for equation \eqref{rescaledequation} on $\bar{\mathcal{B}}_I$.
\begin{theorem}\label{CauchyProblem}
The Cauchy problem for the rescaled equation \eqref{rescaledequation} on $\bar{\mathcal{B}}_I$ is well-posed in the sense that, for any $(\hat{\psi}_0,\hat{\psi}_1) \in \mathcal{H}(\Sigma_0)$, there exists a unique solution $\hat{\psi} \in \mathcal{D}'(\bar{\mathcal{B}}_I)$ to \eqref{rescaledequation} such that
$$(\hat\psi,\partial_t\hat\psi) \in \mathcal{C}(\mathbb{R}_t;\cup_{t\in \mathbb{R}}\mathcal{H}_{\Sigma_t}): \, \hat\psi|_{\Sigma_0}=\hat{\psi}_0; \, \partial_t\hat\psi|_{\Sigma_0} = \hat{\psi}_1.$$
Moreover, $\hat{\psi}$ belongs to $H^1_{loc}(\bar{\mathcal B}_I)$\footnote{The Sobolev space $H^s_{loc}(\bar{\mathcal B}_I)\, (0\leq s<+\infty)$ on open sets is defined in \cite[Definition 2]{Ni2016}}.
\end{theorem}
\begin{proof}
We consider arbitrary neighbourhoods $\mathcal{O}$ of future timelike infinity $i^+$ and $\mathcal{V}$ of spacelike infinity $i_0$, lying outside the support of the initial data. The domain $(\bar{\mathcal{B}}_I \setminus (\mathcal{O} \cup \mathcal{V}), \hat{g})$ can be embedded into a globally hyperbolic cylindrical spacetime $(\mathfrak{M}, \mathfrak{g})$ as a spatially bounded and temporally finite domain (see, for instance, \cite{Pha2023}, Theorem 1).

Applying the energy estimates \eqref{ineq3} and \eqref{ineq4}, and following the approach of \cite{CaCho}, we obtain well-posedness for equation \eqref{rescaledequation} in the extended spacetime $(\mathfrak{M}, \mathfrak{g})$, with solution $\widetilde{\psi}$. By local uniqueness and causality, and using in particular finite propagation speed, it follows that equation \eqref{rescaledequation} is also well posed on $(\bar{\mathcal{B}}_I \setminus (\mathcal{O} \cup \mathcal{V}), \hat{g})$, with solution given by the restriction $\hat{\psi} = \widetilde{\psi}|_{\bar{\mathcal{B}}_I \setminus (\mathcal{O} \cup \mathcal{V})}$.
\end{proof}

The well-posedness established in Theorem \ref{CauchyProblem} allows us to define the trace operator on the conformal boundaries as follows:
\begin{definition}(Trace operator). Let $(\hat{\psi}_0, \hat{\psi}_1)\in \mathcal{C}_0^\infty(\Sigma_0)\times \mathcal{C}_0^\infty(\Sigma_0)$. Consider the solution of equation \eqref{rescaledequation}, $\hat{\psi} \in \mathcal{C}^\infty(\bar{\mathcal{B}_I})$ such that
$$\hat{\psi}|_{\Sigma_0} = \hat{\psi}_0, \, \partial_t \hat{\psi}_1|_{\Sigma_0} = \hat{\psi}_1.$$
We define the trace operator $\mathbb{T}^+$ from $\mathcal{C}_0^\infty(\Sigma_0)\times \mathcal{C}_0^\infty(\Sigma_0)$ to $\mathcal{C}_0^\infty(\mathfrak{H}^+)\times \mathcal{C}_0^\infty(\scri^+)$ as follows:
$$\mathbb{T}^+(\hat{\psi}_0, \hat{\psi}_1) = (\hat{\psi}|_{\mathfrak{H}^+}, \hat{\psi}|_{\scri^+}).$$
\end{definition}
\begin{figure}[H]
\begin{center}
\includegraphics[scale=0.7]{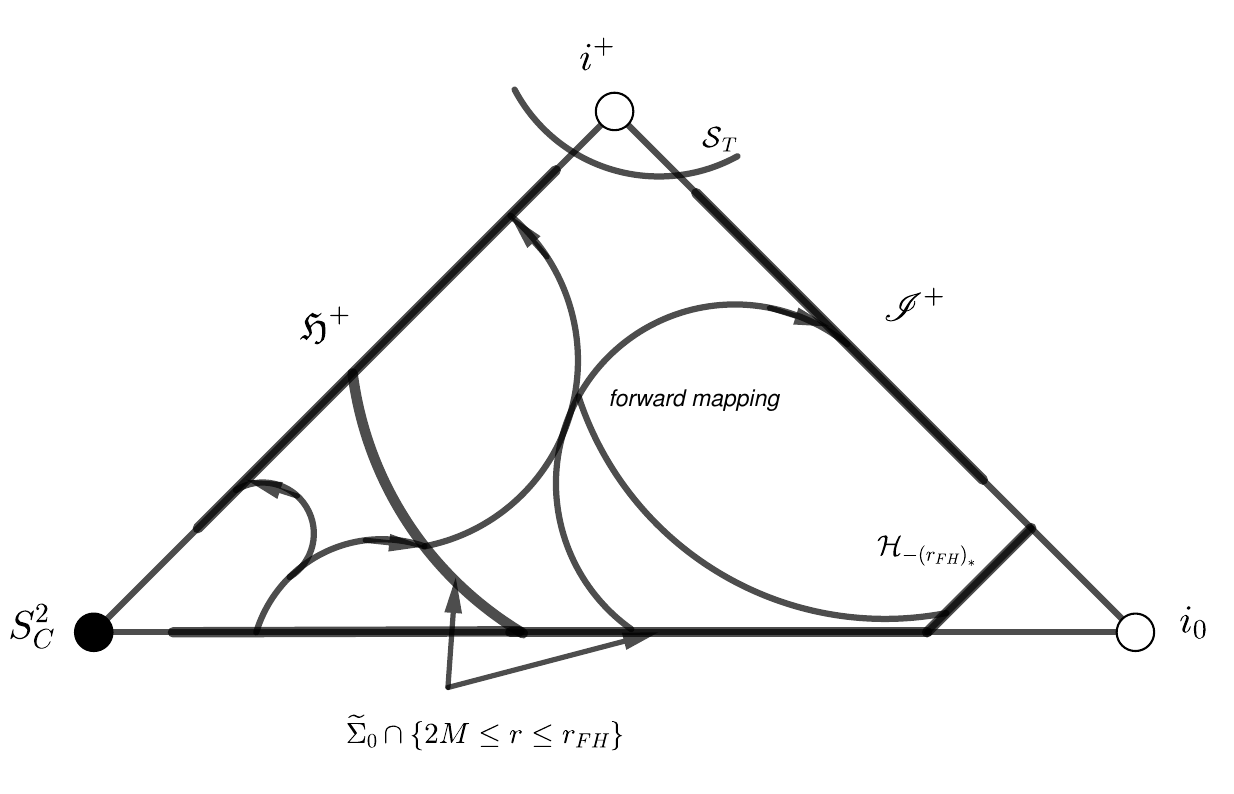}
\caption{Forward mapping from data on $\Sigma_0$}
\end{center}
\end{figure}
\begin{theorem}\label{injective}
The trace operator $\mathbb{T}^+$ is injective.
\end{theorem}
\begin{proof}
Assume that there exists another initial data $(\widetilde{\psi}_0, \widetilde{\psi}_1)$ such that equation \eqref{rescaledequation} admits a solution $\widetilde{\psi}$ and
$\mathbb{T}^+(\widetilde{\psi}_0,\widetilde{\psi}_1) = (\hat{\psi}|_{\mathfrak{H}^+}, \hat{\psi}|_{\scri^+})$.
Then the difference $\hat{\psi} - \widetilde{\psi}$ corresponding to the initial data $(\hat{\psi}_0 - \widetilde{\psi}_0, \hat{\psi}_1 - \widetilde{\psi}_1)$ satisfies the following nonlinear wave equation.
\begin{equation}\label{Lipschitz}
\Box_{\hat g}(\hat\psi - \widetilde{\psi}) + 2MR(\hat\psi - \widetilde{\psi}) + (\hat\psi^2 + \hat{\psi}\widetilde{\psi}+ \widetilde{\psi}^2)(\hat\psi - \widetilde{\psi}) = 0.
\end{equation}
Noting that the restriction of $\hat{\psi} - \widetilde{\psi}$ to the conformal boundary $\mathfrak{H}^+ \cup \mathscr{I}^+$ is identically zero.

As in the proof of Theorem \ref{CauchyProblem}, we consider equation \eqref{Lipschitz} on a bounded and temporally finite domain of the globally hyperbolic cylindrical spacetime $(\mathfrak{M}, \mathfrak{g})$. We assume that $\bar{\mathcal{B}}_I \setminus (\mathcal{O} \cup \mathcal{V})$ is extended to a domain $\mathfrak{B}$, which is foliated by $\{\mathcal{U}_\tau\}_{0 \le \tau \le T'}$ (where $\mathcal{U}\tau$ is the extension of $\mathcal{S}_\tau$), and that the null conformal boundary $\mathfrak{H}^+ \cup \mathscr{I}^+$ is extended to $\mathcal{C}^+$. Note that the time variable is bounded, $\tau \le T$, since we have cut off the neighbourhood $\mathcal{O}$ of $i^+$ by the hypersurface $\mathcal{S}_T$ (with $0 < T < T'$), which yields a temporally finite foliated domain. Therefore, proceeding as in the proof of Theorem 3.5 in \cite{Jo2012}, we obtain that
\begin{equation}\label{enLip1}
\hat{\mathbb{E}}^{\partial_t}_{\mathcal{C}^+} [\hat\psi-\widetilde\psi]\leq C_1 \hat{\mathbb{E}}^{\partial_t}_{\mathcal{U}_0}[\hat\psi-\widetilde\psi]
\end{equation}
as well as the reverse inequality
\begin{equation}\label{enLip2}
\hat{\mathbb{E}}^{\partial_t}_{\mathcal{U}_0}[\hat\psi-\widetilde\psi] \leq C_2\hat{\mathbb{E}}^{\partial_t}_{\mathcal{C}^+} [\hat\psi-\widetilde\psi],
\end{equation}
where the constants $C_1$ and $C_2$ depend on the energies $\hat{\mathbb{E}}^{\partial_t}_{\mathcal{U}_0}[\hat\psi]$ and $\hat{\mathbb{E}}^{\partial_t}_{\mathcal{U}_0}[\widetilde\psi]$. 
The estimate \eqref{enLip2}, together with the vanishing of $\hat{\psi} - \widetilde{\psi}$ on the extension null boundary $\mathcal{C}^+$, implies that
$\hat{\psi}_0 = \widetilde{\psi}_0$ and $\hat{\psi}_1 = \widetilde{\psi}_1$ on the extended hypersurface $\mathcal{U}_0$. By restriction, the same equalities hold on $\mathcal{S}_0$. Moreover, we may choose $r_{FH}$ sufficiently large so that the point $(0, r_{FH}, \omega^2)$ lies outside the support of the initial data.

Therefore, it remains to prove the equality of the initial data on $\Sigma_0$ for $2M \le r \le \dfrac{5M}{2}$, given that
$\hat{\psi} = \widetilde{\psi} \quad \text{on} \quad \mathfrak{H}^+ \cup \bigl(\widetilde{\Sigma}_0 \cap \{ 2M \le r \le \tfrac{5M}{2} \} \bigr)$.
Since we have a bounded domain enclosed by the hypersurfaces $\Sigma_0$, $\mathfrak{H}^+$, and $\widetilde{\Sigma}_0 \cap \{ 2M \le r \le \tfrac{5M}{2} \}$, which can be embedded into a globally hyperbolic cylindrical spacetime as a finite domain, we can extend the proof of \cite[Theorem 3.5]{Jo2012} to establish the following two-sided energy estimates.
\begin{equation}\label{enLip3}
\hat{\mathbb{E}}^{\partial_t}_{\Sigma_0\cap \{ 2M\leq r\leq \frac{5M}{2}\}}[\hat{\psi}-\widetilde{\psi}] \leq C_3 \left( \left( \hat{\mathbb{E}}^{\partial_t}_{\mathfrak{H}^+}[\hat{\psi}-\widetilde{\psi}]\right)^3 + \hat{\mathbb{E}}^{\partial_t}_{\mathfrak{H}^+}[\hat{\psi}-\widetilde{\psi}] + \hat{\mathbb{E}}^{\partial_t}_{\widetilde{\Sigma}_0 \cap \{2M\leq r\leq \frac{5M}{2}\})}[\hat{\psi}-\widetilde{\psi};3] \right)
\end{equation}
and the reverse estimate
\begin{eqnarray}\label{enLip4}
\hat{\mathbb{E}}^{\partial_t}_{\mathfrak{H}^+\cap \{ v\leq (\frac{5M}{2})_*\}}[\hat{\psi}-\widetilde{\psi}] + \hat{\mathbb{E}}^{\partial_t}_{\widetilde{\Sigma}_0\cap \{ 2M\leq r\leq \frac{5M}{2}\}}[\hat{\psi}-\widetilde{\psi}] &\leq& C_4 \left( \left( \hat{\mathbb{E}}^{\partial_t}_{\Sigma_0 \cap \{2M\leq r\leq \frac{5M}{2} \}}[\hat{\psi}-\widetilde{\psi}]\right)^3 \right.\cr
&&\left.\hspace{2cm}+ \hat{\mathbb{E}}^{\partial_t}_{\Sigma_0\cap \{2M\leq r\leq \frac{5M}{2} \}}[\hat{\psi}-\widetilde{\psi}]  \right).
\end{eqnarray}
Finally, the energy estimate \eqref{enLip3} yields the equalities
$\hat{\psi}_0 = \widetilde{\psi}_0 \quad \text{and} \quad \hat{\psi}_1 = \widetilde{\psi}_1$
on $\Sigma_0$ for $2M \le r \le \dfrac{5M}{2}$. The proof is therefore complete.
\end{proof}

We now define the scattering data space:
\begin{definition}
The scattering data space $\mathcal{H}^+$ is defined as the completion of 
$\mathcal{C}_0^\infty(\mathfrak{H}^+) \times \mathcal{C}_0^\infty(\scri^+)$ with respect to the norm
$$\left\| (\xi,\zeta) \right\|_{\mathcal{H}^+} = \left(\int_{\mathfrak{H}^+}(\partial_v\xi)^2 \d v\d^2\omega + \int_{\scri^+}(\partial_u\zeta)^2 \d u \d^2\omega \right)^{1/2}.$$
This means that
$$\mathcal{H}^+ \simeq \dot{H}^1(\mathbb{R}_v; L^2(S^2_\omega)) \times \dot{H}^1(\mathbb{R}_u; L^2(S^2_\omega)).$$
\end{definition}
\begin{theorem}\label{Trace}
The trace operator extends uniquely to a bounded linear map from $\mathcal{H}$ to $\mathcal{H}^+$. The extended operator is locally Lipschitz.
\end{theorem}
\begin{proof}
The locally Lipschitz continuity of $\mathbb{T}^+$ is a consequence of the estimates \eqref{enLip1} and \eqref{enLip4} derived in the proof of Theorem~\ref{injective}.
\end{proof}

\subsection{Goursat problem and conformal scattering operator}
In order to obtain the conformal scattering operator, we need to show that the trace operator is surjective. This amounts to solving the Goursat problem for equation \eqref{rescaledequation} with initial data prescribed on the conformal boundary $\mathfrak{H}^+ \cup \scri^+$ (respectively, $\mathfrak{H}^- \cup \scri^-$).
As a starting point, Hörmander \cite{Ho1990} solved the Goursat problem for second-order scalar linear wave equations with first-order regular potentials on spatially compact spacetimes. Subsequently, Nicolas \cite{Ni2006} extended Hörmander’s results to establish the well-posedness of the Goursat problem for linear wave equations with continuous coefficients in the first-order terms and locally $L^\infty$-coefficients in the zeroth-order terms, on slightly more regular metrics. Moreover, Joudioux \cite{Jo2012,Jo2019} treated the Goursat problem for nonlinear wave equations on asymptotically simple spacetimes.
The Goursat problem has also been studied in other contexts, for instance for Dirac and Maxwell equations in \cite{MaNi2004,Mo2019,Pha2022,Pha22}, and for scalar, vectorial, and spinor wave equations in \cite{Ni2016,Pha2020,Pha22,Pha2024}.
In the following, we apply and extend the results of \cite{Jo2012,Jo2019} to establish the well-posedness of the Goursat problem for equation \eqref{rescaledequation} on $(\bar{\mathcal{B}}_I, \hat{g})$.
\begin{theorem}\label{Goursat}
The Goursat problem for equation \eqref{rescaledequation} is well-posed, that is, for any initial data 
$(\xi, \zeta) \in C_0^\infty(\mathfrak{H}^+) \times C_0^\infty(\scri^+)$
and any spacelike foliation $\{ \mathcal{W}_t \}_{t}$ of $\bar{\mathcal{B}}_I$, there exists a unique solution to \eqref{rescaledequation} such that
$$(\hat\psi,\partial_t\hat\psi)\in \mathcal{C}(\mathbb{R}_+;\cup_{t\geq 0}\mathcal{H}(\mathcal{W}_t)) \hbox{  and  } \mathbb{T}^+(\hat\psi|_{\Sigma_0}, \partial_t \hat\psi|_{\Sigma_0}) = (\xi,\zeta).$$
Consequently, the trace operator $\mathbb{T}^+ : \mathcal{H} \to \mathcal{H}^+$ is surjective. Furthermore, the inverse operator $(\mathbb{T}^+)^{-1}$ extends uniquely to a bounded linear and locally Lipschitz map from $\mathcal{H}^+$ to $\mathcal{H}$.
\end{theorem}
\begin{proof}
Let $\widetilde{\mathcal{S}}_0$ be a hypersurface which coincides with $\mathcal{S}_0$ for $r \ge \frac{5M}{2}$ and consists of the null hypersurface
$\mathcal{L} = \left\{ v = \left( \frac{5M}{2} \right)_* \right\}$.
We split the well-posedness analysis of the Goursat problem into the two domains
$\mathcal{I}^+(\Sigma_0) \cap \left\{ v \le \left( \frac{5M}{2} \right)_* \right\}
\quad \text{and} \quad
\mathcal{I}^+(\Sigma_0) \cap \left\{ v \ge \left( \frac{5M}{2} \right)_* \right\}$.

First, by applying the results of \cite[Theorem~3.2 and Proposition~6.1]{Jo2019} (see also \cite[Theorem~3.13]{Jo2012}), we obtain the well-posedness of the Goursat problem on 
$\mathcal{I}^+(\widetilde{\mathcal{S}}_0) \setminus (\mathcal{O} \cup \mathcal{V})$.
More precisely, for initial data 
$(\xi, \zeta) \in \mathcal{C}_0^\infty(\mathfrak{H}^+) \times \mathcal{C}_0^\infty(\scri^+)$,
where we cut off a neighbourhood $\mathcal{O}$ of $i^+$ by a spacelike hypersurface $\mathcal{S}_T$ lying sufficiently far from the support of the initial data (i.e.\ for $T<+\infty$ large enough), and cut-off a neighbourhood $\mathcal{V}$ of $i_0$, which is also far from the support of the initial data, equation \eqref{rescaledequation} admits a unique solution $\hat{\psi}^a$ satisfying
$$\hat{\psi}^a \in {\mathcal {C}}(\mathbb{R}_+;\, \cup_{\tau \geq 0}H^1(\mathcal{W}_\tau))\cap \mathcal{C}^1(\mathbb{R}_+; \, \cup_{\tau \geq 0}L^2(\mathcal{W}_\tau))$$
and
$$\hat{\psi}^a|_{\scri^+} = \zeta, \, \hat{\psi}^a|_{\mathfrak{H}^+\cap \{v\geq (\frac{5M}{2})_* \}}= \xi.$$
Here $\{ \mathcal{W}_\tau \}_{\tau \ge 0}$ (with $\mathcal{W}_0 = \widetilde{\mathcal{S}}_0$) is a foliation of $\mathcal{I}^+(\widetilde{\mathcal{S}}_0)$. Moreover, we can obtain energy estimates relating the energy of $\hat{\psi}^a$ on $\widetilde{\mathcal{S}}_0$ to the energy on 
$\mathfrak{H}^+_{v \ge \left( \frac{5M}{2} \right)_*} \cup \scri^+$,
as in Corollary~\ref{1}.

\begin{figure}[H]
\begin{center}
\includegraphics[scale=0.7]{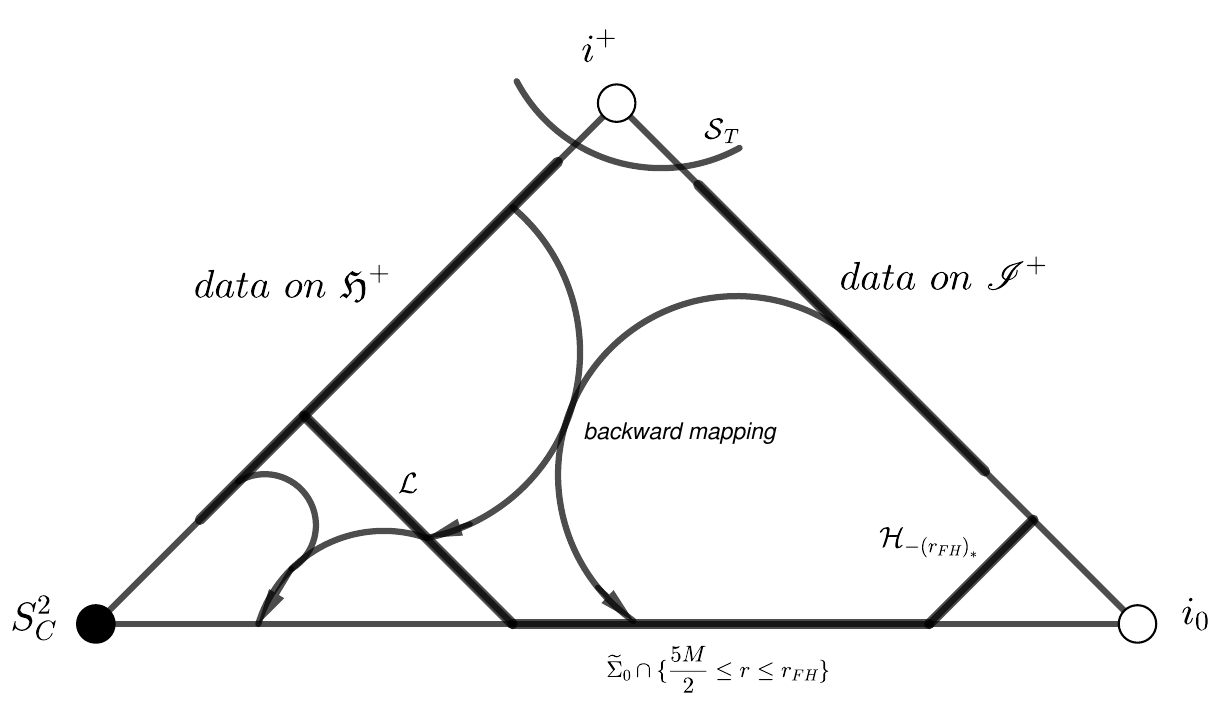}
\caption{Backward mapping from data on $\mathfrak{H}^+\cup \mathscr{I}^+$}
\end{center}
\end{figure} 

Consider the initial data 
$\bigl( \hat{\psi}^a \big|_{\mathcal{L}}, \, \xi \big|_{v \le \left( \frac{5M}{2} \right)_*} \bigr)$.
By applying again \cite[Theorem~3.2 and Proposition~6.1]{Jo2019}, the Goursat problem for equation \eqref{rescaledequation} is well-posed. We denote the corresponding solution by $\hat{\psi}^b$. 
Moreover, we can obtain energy estimates for $\hat{\psi}^b$ relating the total energy of the field through 
$\mathfrak{H}^+_{v \le \left( \frac{5M}{2} \right)_*} \cup \mathcal{L}$
to the energy on 
$\Sigma_0 \cap \left\{ 2M \le r \le \frac{5M}{2} \right\}$,
including the estimates as in Theorem~\ref{injective}.

The solution $\hat{\psi}$ of the Goursat problem for equation \eqref{rescaledequation} is obtained by gluing together the two solutions $\hat{\psi}^a$ and $\hat{\psi}^b$ of the Goursat problem for \eqref{rescaledequation} on the domains 
$\left\{ v \ge \left( \frac{5M}{2} \right)_* \right\}
\quad \text{and} \quad
\left\{ v \le \left( \frac{5M}{2} \right)_* \right\}$,
respectively. Moreover, if we denote
$$\hat{\psi}_0 = \hat{\psi}|_{\Sigma_0},\,\, \hat{\psi}_1 = \partial_t\hat{\psi}|_{\Sigma_0},$$
we obtain
$$(\hat{\psi}_0,\hat{\psi}_1)\in \mathcal{H} \hbox{  and  } (\xi,\zeta) = \mathbb{T}^+(\hat{\psi}_0,\hat{\psi}_1).$$
Therefore, the range of $\mathbb{T}^+$ contains $\mathcal{C}_0^\infty(\mathfrak{H}^+) \times \mathcal{C}_0^\infty(\scri^+)$ and $\mathbb{T}^+$ is surjective. The inverse operator $(\mathbb{T}^+)^{-1}$ is locally Lipschitz, by virtue of the energy estimates for $\hat{\psi}^a$ and $\hat{\psi}^b$ and by the same argument as in the proof of Theorem~\ref{injective}.
\end{proof}
Similarly to the operator $\mathbb{T}^+$, we define the past trace operator 
$\mathbb{T}^- : \mathcal{H} \to \mathcal{H}^-$, {associated with} solutions of equation \eqref{rescaledequation} on $\mathcal{I}^-(\Sigma_0)$. Combining the two operators $\mathbb{T}^\pm$, we now define the conformal scattering operator for equation \eqref{rescaledequation} as follows:
\begin{definition}
We define the scattering operator $\mathbb{S}$ as the map that associates to the past scattering data the corresponding future scattering data, namely
$$\mathbb{S} = \mathbb{T}^+ \circ (\mathbb{T}^-)^{-1} : \mathcal{H}^- \to \mathcal{H}^+.$$
The scattering operator $\mathbb{S}$ is invertible and locally Lipschitz.
\end{definition}
\noindent
{\bf Data availability.} No datasets were generated or analysed during the current study.\\
{\bf Ethics declarations.}
The author declares no conflict of interest.\\

\end{document}